\documentclass[bjps,preprint]{imsart}
\pdfoutput=1
\RequirePackage[OT1]{fontenc}
\usepackage{amsthm,amsmath,natbib}

\usepackage{amsmath,amsfonts,amssymb,amsthm,color}
\usepackage{enumerate}
\usepackage{dsfont}

\startlocaldefs
\numberwithin{equation}{section}
\theoremstyle{plain}

\theoremstyle{plain}
\newtheorem{theorem}{Theorem}[section]

\theoremstyle{remark}
\newtheorem{rem}[theorem]{Remark}

\newtheorem{ex}[theorem]{Example}

\theoremstyle{definition}
\newtheorem{defi}[theorem]{Definition}

\theoremstyle{plain}

\newtheorem{lem}[theorem]{Lemma}
\newtheorem{prop}[theorem]{Proposition}

\numberwithin{equation}{section}

\newcommand\norm[1]{\left|\!\left| #1\right|\!\right|}

\newcommand\BIP{\operatorname{BIP}}

\newcommand\Vertt{|\!|}

\newcommand{\M}{\mathcal{M}}

\newcommand{\R}{\mathds{R}}

\newcommand{\Exp}{\mathds{E}}

\newcommand{\Prob}{\mathds{P}}

\newcommand{\Lin}{\mathcal{L}}
\newcommand{\abs}[1]{\left\vert#1\right\vert}
\newcommand{\set}[1]{\left\{#1\right\}}

\newcommand{\bE}{\mathbf{E}}
\newcommand{\bH}{\mathbf{H}}

\newcommand{\B}{\mathcal{B}}
\newcommand{\A}{\mathcal{A}}

\newcommand{\Fil}{\mathcal{F}}

\endlocaldefs

\begin{document}

\begin{frontmatter}

\title{A note on space-time H\"{o}lder regularity of mild solutions to stochastic Cauchy problems in $L^p$-spaces}

\runtitle{H\"{o}lder regularity of stochastic Cauchy problems}

\begin{aug}

\author{\fnms{Rafael} \snm{Serrano}\thanksref{a}\corref{}\ead[label=e1]{rafael.serrano@urosario.edu.co}}

\affiliation[a]{Universidad del Rosario\\Bogot\'{a}, Colombia}

\address[a]{Universidad del Rosario\\Calle 12c No. 4-69\\Bogot\'{a}, Colombia\\\printead{e1}}

\runauthor{Rafael Serrano}

\end{aug}

\begin{abstract}
This paper revisits the H\"{o}lder regularity of mild solutions of parabolic stochastic Cauchy problems in Lebesgue spaces $L^p(\mathcal{O}),$ with $p\geq 2$ and $\mathcal{O}\subset\R^d$ a bounded domain. We find conditions on $p, \beta$ and $\gamma$ under which the mild solution has almost surely trajectories in $\mathcal{C}^\beta([0,T];\mathcal{C}^\gamma(\bar{\mathcal{O}})).$ These conditions do not depend on the Cameron-Martin Hilbert space associated with the driving cylindrical noise. The main tool of this study is a regularity result for stochastic convolutions in M-type 2 Banach spaces by \cite{brz1}.
\end{abstract}



\end{frontmatter}

\section{Introduction}
\label{intro}

Let $d\geq 1$ and let $\mathcal{O}\subset\R^d$ be a bounded domain. Let $\bH$ be a separable Hilbert space. In this short note we revisit the spatial and temporal H\"{o}lder regularity of mild solutions to stochastic Cauchy problems in $L^p(\mathcal{O})$ of the form
\begin{equation}\label{scp-intro}
\begin{split}
du(t)+A_pu(t)\,dt&=G(t)\,dW(t), \ \ t\in[0,T]\\
u(0)&=0
\end{split}
\end{equation}
where $A_p$ is the realization in $L^p(\mathcal{O})$ of a second-order differential operator with smooth coefficients, $G(\cdot)$ is an $\Lin(\bH,L^p(\mathcal{O}))$-valued process and $W(\cdot)$ is an $\bH$-cylindrical Wiener process.

Space-time regularity of  linear (affine) stochastically forced evolution equations driven by cylindrical noise has been studied by several authors using the mild solution approach in Hilbert (see, e.g. Section 5.5. of \cite{dpz1}, Section 3 of  \cite{cerrai0}) and Banach spaces (see, e.g. \cite{brz1}, Section 3.2 of \cite{brzgat}, and \cite{dvnw}).

In this paper, we find conditions on $p, \beta$ and $\gamma$ under which the mild solution to (\ref{scp-intro}) exists and has almost surely trajectories in $\mathcal{C}^\beta([0,T];\mathcal{C}^\gamma(\bar{\mathcal{O}})),$ see Proposition \ref{prop2} below. It is worth noting that these conditions do not depend on the Hilbert space $\bH,$ unlike nearly all existing results in the literature.

Following completion of the first draft version of this note, the author became aware of a space-time regularity result in a recent article by  \cite{vnvw2012} (see Theorem 1.2-(1) in that article) which seems comparable to our main result. However, their approach is much more involved as it is largely based on McIntosh's $H^\infty$-functional calculus  and R-boundedness techniques. The approach in this short note is simpler as it relies only on regularity results for stochastic convolutions in M-type 2 Banach spaces by \cite{brz1}.

We argue that, using the factorization method introduced by \cite{dapratoetal} and fixed-point arguments as in \cite{brz1}, this result can be easily generalized to mild solutions of semi-linear stochastic PDEs with multiplicative cylindrical noise, linear growth coefficients and zero Dirichlet-boundary conditions, as well as Neumann-type boundary conditions.

Let us briefly describe the contents of this paper. In section 2 we outline the construction of the stochastic integral and stochastic convolutions in M-type 2 Banach spaces with respect to a cylindrical Wiener process. For the details and proofs we refer to \cite{brz0,brz1,brz2} and the references therein.

In section 3 we state and prove our main result on H\"{o}lder space-time regularity for mild solutions of equation (\ref{scp-intro}). We apply this result to a linear stochastic PDE with a noise term that is ``white" in time but ``colored" in the space variable. Such noise terms are particularly relevant in $d$ dimensions with $d>1.$ We also illustrate how the main result can be generalized to incorporate stochastic PDEs with linear operators given as the fractional power of second-order partial differential operators.







\section{Stochastic convolutions in M-type 2 Banach spaces}
Let $(\Omega ,\Fil,\Prob)$ be a probability space endowed with a filtration $\mathds{F}=\{\Fil_t\}_{t\ge 0}$ and let $\left(\bH,[\cdot,\cdot]_\bH\right)$ denote a separable Hilbert space.
\begin{defi}
A family $W(\cdot)=\{W(t)\}_{t\geq 0}$ of bounded linear operators from $\bH$ into $L^2(\Omega;\R)$ is called an $\bH$-\emph{cylindrical Wiener process} (with respect to the filtration $\mathds{F})$ iff the following hold
\begin{enumerate}[(i)]
\item $\Exp\, W(t)y_1 W(t)y_2 = t[y_1 ,y_2]_\bH$ for all $t\ge 0$ and $y_1 ,y_2 \in\bH$.

\item For each $y\in\bH$, the process $\{W(t)y\}_{t\geq 0}$ is a standard one-dimensional Wiener process with respect to $\mathds{F}.$
\end{enumerate}
\end{defi}
For $q\geq 1,$ $T\in(0,\infty)$ and a Banach space $(V,\abs{\cdot}_V)$,  let $\M^q(0,T;V)$ denote the space of (classes of equivalences of) $\mathds{F}$-progressively measurable processes $\Phi:[0,T]\times\Omega\to V$ such that
\[
\norm{\Phi}^q_{\M^q(0,T;V)}:=\Exp\int_0^T\abs{\Phi(t)}^q_V\,dt<\infty.
\]
This is a Banach space when endowed with the norm $\norm{\cdot}_{\M^q(0,T;V)}.$

\begin{defi}
A process $\Phi(\cdot)$ with values in $\Lin(\bH,\bE)$ is said to be \emph{elementary} (with respect to the filtration $\mathds{F}$) if there exists a partition $0=t_0<t_1\cdots<t_N=T$ of $[0,T]$ such that
\[
\Phi(t)=\sum_{n=0}^{N-1}\sum_{k=1}^K \mathbf{1}_{[t_n,t_{n+1})}(t)[e_k,\cdot]_\bH\xi_{kn}, \ \ \ t\in [0,T].
\]
where $(e_k)_{k\geq 1}$ is an orthonormal basis of $\bH$ and $\xi_{kn}$ is an $\Fil_{t_n}-$measurable $\bE-$valued random variable , for $n=0,1,\dots,N-1, \ k=1,\ldots,K.$ For such processes we define the \emph{stochastic integral} as
\[
I_T(\Phi):=\int_0^T\Phi(t)\,dW(t):=\sum_{n=0}^{N-1}\sum_{k=1}^K\left(W(t_{n+1})e_k-W(t_{n})e_k\right)\xi_{kn}.
\]
\end{defi}
\begin{defi}
Let $(\gamma_k)_k$ be a sequence of real-valued standard Gaussian random variables. A bounded linear operator $R:\bH\to\bE$ is said to be $\gamma-$\emph{radonifying} iff there exists an orthonormal basis $(e_k)_{k\geq 1}$ of $\bH$ such that the sum $\sum_{k\geq 1}\gamma_k Re_k$ converges in $L^2(\Omega;\bE).$
\end{defi}
We denote by $\gamma(\bH,\bE)$ the class of $\gamma-$radonifying operators from $\bH$ into $\bE$,  which is a Banach space equipped with the norm
\[
\norm{R}^2_{\gamma(\bH,\bE)}:=\Exp\Biggl[\biggl|\sum_{k\geq 1}\gamma_k Re_k\biggr|^2_\bE \Biggr] , \ \ \ \ R\in \gamma(\bH,\bE).
\]
The above definition is independent of the choice of the orthonormal basis $(e_k)_{k\geq 1}$ of $\bH.$ Moreover, $\gamma(\bH,\bE)$ is continuously embedded into $\Lin(\bH,\bE)$ and is an operator ideal in the sense that if $\bH'$ and $\bE'$ are Hilbert and Banach spaces, respectively, such that $S_1\in\Lin(\bH',\bH)$ and $S_2\in\Lin(\bE,\bE')$ then $R\in \gamma(\bH,\bE)$ implies $S_2RS_1\in\gamma(\bH',\bE')$ with
\[
\norm{S_2RS_1}_{\gamma(\bH',\bE')}\le \norm{S_2}_{\Lin(\bE,\bE')}\norm{R}_{\gamma(\bH,\bE)}\norm{S_1}_{\Lin(\bH',\bH)}
\]

It can be proved that $R\in \gamma(\bH,\bE)$ iff $RR^*$ is the covariance operator of a centered Gaussian measure on $\B(\bE)$,  and if $\bE$ is a Hilbert space, then $\gamma(\bH,\bE)$ coincides with the space of Hilbert-Schmidt operators from $\bH$ into $\bE$ (see e.g. \cite{vn0} and the references therein). The following is also a very useful characterization of $\gamma-$radonifying operators in the case that $\bE$ is a $L^p-$space,
\begin{lem}[\cite{vnvw}, Lemma 2.1]\label{gammalp}
Let $(S,\mathfrak{A},\rho)$ be a $\sigma-$finite measure space and let $p\geq 1.$ Then, for an operator $R\in\Lin(\bH,L^p(S))$ the following assertions are equivalent
\begin{enumerate}
  \item $R\in\gamma(\bH,L^p(S))$.
  \item There exists a function $g\in L^p(S)$ such that for all $y\in\bH$ we have
  \[
  \abs{(Ry)(\xi)}\le \abs{y}_\bH\cdot g(\xi), \ \ \ \rho-\mbox{a.e.} \ \xi\in S.
  \]
\end{enumerate}
If either of these two assertions holds true, there exists a constant $c>0$ such that $\norm{R}_{\gamma(\bH,L^p(S))}\le c\abs{g}_{L^p(S)}.$
\end{lem}

\begin{defi}
A Banach space $\bE$ is said to be of \emph{martingale type} $2$ (and we write $\bE$ is \emph{M-type} $2$) iff there exists a constant $C_2>0$ such that
\begin{equation}\label{mtype2}
\sup_{n} \mathbb{E} | M_{n} |_\bE ^{2} \le C_2 \sum_{n} \mathbb{E}  | M_{n}-M_{n-1} |_\bE ^{2}
\end{equation}
for any $\bE-$valued discrete martingale $\{M_{n}\}_{n\in \mathds{N}}$ with $M_{-1}=0.$
\end{defi}

\begin{ex}
Hilbert spaces and Lebesgue spaces $L^{p}(\mathcal{O}),$ with $p\geq 2$ and $\mathcal{O}\subset\R^d$ a bounded domain, are examples of M-type $2$ Banach spaces.
\end{ex}

If $\bE$ is a M-type 2 Banach space, it is easy to show (see e.g. \cite{dett}) that the stochastic integral $I_T(\Phi)$ for elementary processes $\Phi(\cdot)$ satisfies
\begin{equation}\label{burk0}
\Exp\abs{I_T(\Phi)}_\bE^2\le C_2 \Exp\int_0^T\norm{\Phi(s)}_{\gamma(\bH,\bE)}^2\,ds
\end{equation}
where $C_2$ is the same constant in (\ref{mtype2}). Since the set of elementary processes is dense in $\M^2(0,T;\gamma(\bH,\bE)),$ see e.g. Lemma 18 in Chapter 2 of \cite{neidhardt}, by (\ref{burk0}) the linear mapping $I_T$ extends to a bounded linear operator from $\M^2(0,T;\gamma(\bH,\bE))$ into $L^2(\Omega;\bE).$ We denote this operator also by $I_T.$

Finally, for each $t\in [0,T]$ and $\Phi\in \M^2(0,T;\gamma(\bH,\bE))$,  we define
\[
\int_0^t\Phi(s)\,dW(s):=I_T(\mathbf{1}_{[0,t)}\Phi).
\]

\begin{defi}
Let $A$ be a linear operator on a Banach space $\bE.$ We say that $A$ is \emph{positive} if it is closed, densely defined, $(-\infty,0]\subset \rho(A)$ and there exists $C\geq 1$ such that
\[
\Vertt (\lambda I+A)^{-1}\Vertt_{\Lin(\bE)} \le  \frac{C}{1+\lambda}, \ \ \ \text{ for all }\lambda\geq 0.
\]
\end{defi}

It is well known that if $A$ is a positive operator on $\bE$,  then $A$ admits (not necessarily bounded) fractional powers $A^z$ of any order $z\in\mathds{C},$ see e.g. \cite[Chapter III, Section 4.6]{amann}. Recall that, in particular, for $\abs{\Re z}\le 1$ the fractional power $A^z$ is defined as the closure of the linear mapping
\begin{equation}\label{Az}
D(A)\ni x\mapsto \frac{\sin \pi z}{\pi z}\int_0^{+\infty} t^z(tI+A)^{-2}Ax\,dt\in\bE,
\end{equation}
Moreover, if $\Re z\in (0,1)$,  then $A^{-z}\in\Lin(\bE)$ and we have
\[
A^{-z}x=\frac{\sin \pi z}{\pi}\int_0^{+\infty}t^{-z}(tI+A)^{-1}x\,dt.
\]
see e.g. \cite[p. 153]{amann}.

\begin{defi}
The class $\BIP(\theta,\bE)$ of operators with \emph{bounded imaginary powers} on $\bE$ with parameter $\theta\in[0,\pi)$ is defined as the class of positive operators $A$ on $\bE$ with the property that $A^{is}\in \Lin(\bE)$ for all $s\in\R$ and there exists a constant $K>0$ such that
\begin{equation}
\Vertt A^{is} \Vertt_{\Lin(\bE)} \le K e^{\theta |s|}, \; s \in \R.
\label{2.1}
\end{equation}
\end{defi}
We denote $\BIP^-(\theta,\bE):=\cup_{\sigma\in(0,\theta)}\BIP(\sigma,\bE).$ The following is the main assumption for the rest of this note
\begin{equation}\label{bip}
A\in\BIP^-(\pi/2,\bE).
\end{equation}
Under this assumption, the linear operator $-A$ generates an (uniformly bounded) analytic $C_0-$semigroup $(S_t)_{t\geq 0}$ on $\bE,$ see e.g. Theorem 2 in \cite{pruesohr1}.

\begin{ex}\label{ex0}
Let $\mathcal{O}$ be a bounded domain in $\R^d$ with smooth boundary and let $\mathcal{A}$ denote the second-order elliptic differential operator
\[
(\mathcal{A}u)(\xi):=-\sum_{i,j=1}^d a_{ij}(\xi)\frac{\partial^2 u}{\partial \xi_i\partial \xi_j}+\sum_{i=1}^d b_{i}(\xi)\frac{\partial u}{\partial \xi_i} + c(\xi)u(\xi), \ \ u\in\mathcal{C}^2(\mathcal{O}), \ \ \xi\in\mathcal{O},
\]
with coefficients $a,b$ and $c$ satisfying the following conditions
\begin{itemize}
  \item[(i)] $a(\xi)=(a_{ij}(\xi))_{1\le i,j\le d}$ is a real-valued symmetric matrix for all $\xi\in\mathcal{O},$ and there  exists $a_0>0$ such that
\[
a_0\le\sum_{i,j=1}^d a_{ij}(\xi)\lambda_i\lambda_j\le \frac{1}{a_0}, \ \mbox{ for all }\xi\in\mathcal{O}, \ \lambda\in\R^d, \ \abs{\lambda}=1.
\]

\item[(ii)] $a_{ij}\in\mathcal{C}^{\alpha}(\bar{\mathcal{O}})$ for some $\alpha\in(0,1).$

  \item[(iii)] $b_i\in L^{k_1}(\mathcal{O})$ and $c\in L^{k_2}(\mathcal{O}),$ for some $k_1>d$ and $k_2>d/2.$
\end{itemize}
For $p>1$ and $\nu\geq 0,$ let $A_{p,\nu}$ denote the realization of $\mathcal{A}+\nu I$ in $L^p(\mathcal{O})$, that is,
\begin{equation}\label{Aqdef}
\begin{split}
A_{p,\nu}u&:=\mathcal{A}u+\nu u\\
D(A_{p,\nu})&:=W^{2,p}(\mathcal{O})\cap W_0^{1,p}(\mathcal{O}).
\end{split}
\end{equation}
By Theorems A and D of \cite{pruesohr2}, if $p\le \min\set{k_1,k_2}$  there exists $\bar\nu\geq 0$ sufficiently large so that $A_{p,\bar\nu}\in\BIP^-(\pi/2,L^p(\mathcal{O})).$

Other examples of operators satisfying main assumption (\ref{bip}) include realizations in $L^p(\mathcal{O})$ of higher order elliptic partial differential operators (see \cite{seeley}), the Stokes operator (see \cite{gigasohr}) and second-order elliptic partial differential operators with Neumann-type boundary conditions (see \cite{sohrthaeter}).
\end{ex}

\begin{theorem}[\cite{brz1}, Theorem 3.2]\label{Th:2.1}
Let $T\in (0,\infty)$ and $A\in\BIP^-(\pi/2,\bE)$ be fixed. Let $\bE$ be an M-type $2$ Banach space and $G(\cdot)$ an $\Lin(\bH,\bE)-$valued stochastic process satisfying
\begin{equation}\label{sigma-eta}
A^{-\sigma} G(\cdot) \in \mathcal{M}^q\left(0,T;\gamma(\bH,\bE)\right).
\end{equation}
for some $q\geq 2$ and $\sigma\in\left[0,\frac{1}{2}\right).$ Then, for each $t\in[0,T],$ we have $S_{t-r}G(r)\in \gamma(\bH,\bE)$ and the map
\[
[0,t]\ni r\mapsto S_{t-r}G(r)\in\gamma(\bH,\bE)
\]
belongs to $\M^q(0,t;\gamma(\bH,\bE)).$ Moreover, the $\bE$-valued process
\begin{equation}\label{2.19}
u(t):=\int^t_0 S_{t-r}G (r)\,dW(r), \ \  t\in [0,T],
\end{equation}
belongs to $\M^q(0,T;\bE)$ and satisfies the estimate
\[
\norm{u(\cdot)}_{\M^q(0,T;\bE)}\le C\norm{A^{-\sigma}G(\cdot)}_{\mathcal{M}^q\left(0,T;\gamma(\bH,\bE)\right)}
\]
for some constant $C$ depending on $\bE,A,T,\sigma$ and $q.$
\end{theorem}

\begin{defi}
For $u_0\in\bE$ given, a process $u(\cdot)\in\M^q(0,T;\bE)$ is called a \emph{mild solution} to the the abstract stochastic  Cauchy problem
\begin{equation}\label{scp0}
\begin{split}
  du(t)+Au(t)\,dt&=G(t)\,dW(t), \ \ t\in [0,T]\\
  u(0)&=u_0
\end{split}
\end{equation}
iff for all $t\in [0,T]$ we have almost surely
\[
u(t)=S_tu_0+\int^t_0 S_{t-r}G (r)\,dW(r).
\]
\end{defi}

\begin{theorem}[\cite{brz1}, Corollary 3.5]\label{Co:2}
Under the  assumptions of Theorem \ref{Th:2.1}, let $\delta$ and $\beta$ satisfy
\begin{equation}
\beta+\delta + \sigma +\frac{1}{q}< \frac{1}{2}.
\label{cond:2}
\end{equation}
Then, there exists a modification of $u(\cdot),$ which we also denote with $u(\cdot),$ that has trajectories almost surely in $\mathcal{C}^\beta([0,T];D(A^\delta))$ and satisfies
\[
\Exp\norm{u(\cdot)}^q_{\mathcal{C}^\beta([0,T];D(A^\delta))}\le C'
\norm{ A^{-\sigma }G(\cdot)}^q_{\mathcal{M}^q\left(0,T;\gamma(\bH,\bE)\right)}
\]
for some constant $C'$ depending on $\bE,T,A,\beta,\delta,\sigma$ and $q.$
\end{theorem}

\begin{rem}\label{nu0}
The above results are still valid if $A+\nu I\in\BIP^-(\pi/2,\bE)$ for some $\nu\geq 0,$ see e.g. \cite[p.192]{brzgat}.
\end{rem}

\section{Main result}
Let $\mathcal{A}$ be the second order differential operator from Example \ref{ex0}, and let $A_p:=A_{p,\bar\nu}$ denote the realization of $\A+\nu I$ on $L^p(\mathcal{O}),$ with $\bar\nu\geq 0$ chosen so that $A_{p,\bar\nu}\in\BIP^-(\frac{\pi}{2},L^p(\mathcal{O})).$ We consider the stochastic Cauchy problem in $L^p(\mathcal{O})$
\begin{equation}\label{scp}
\begin{split}
du(t)+A_pu(t)\,dt&=G(t)\,dW(t), \ t\in[0,T],\\
u(0)&=0.
\end{split}
\end{equation}

\begin{lem}\label{lem1}
Assume $m:=\min\set{k_1,k_2}>\max\set{2,d}$ and
\begin{equation}\label{ineqpm}
p\in\bigl(\max\set{2,d},m\bigr].
\end{equation}
Let $G(\cdot)$ be an $\Lin(\bH,L^p(\mathcal{O}))$-valued process such that
\begin{equation}\label{sigma-eta-2}
G(\cdot)\in\M^q\left(0,T;\Lin(\bH,L^p(\mathcal{O}))\right).
\end{equation}
Then, for any $\sigma\in\bigl(\frac{d}{2p},\frac{1}{2}\bigr),$ $A_p^{-\sigma}G(\cdot)$ is an $\gamma(\bH,L^p(\mathcal{O}))$-valued process and we have
\begin{equation}\label{Aqsigmaeta}
A_p^{-\sigma}G(\cdot)\in\M^q\bigl(0,T;\gamma(\bH,L^p(\mathcal{O}))\bigr).
\end{equation}
\end{lem}

\begin{proof}
By Theorem 1.15.3 in \cite{triebel} we have
\[
D(A_p^\sigma)=[L^p(\mathcal{O}),D(A_p)]_\sigma\subseteq [L^p(\mathcal{O}),W^{2,p}(\mathcal{O})]_\sigma=H^{2\sigma,p}(\mathcal{O}).
\]
with continuous embeddings. Here $[\cdot,\cdot]_\sigma$ denotes complex interpolation and $H^{2\sigma,p}(\mathcal{O})$ denotes the Bessel-potential space of fractional order $2\sigma,$ see e.g. \cite{triebel}.

By the Sobolev embedding theorem, we have
$H^{2\sigma,p}(\mathcal{O})\subset \mathcal{C}(\bar{\mathcal{O}})$ with continuous embedding, and since $\mathcal{O}$ is bounded we also have $\mathcal{C}(\bar{\mathcal{O}})\subset L^{\infty}(\mathcal{O}).$ Let $c_{\sigma,p}>0$ denote the norm of the continuous embedding $D(A_p^\sigma)\subset L^\infty(\mathcal{O})$. Then, for any $y\in \bH$ we have
\begin{align*}
\abs{A_p^{-\sigma}G(t)y}_{L^\infty(\mathcal{O})}
&\le c_{\sigma,p}\abs{A_p^{-\sigma}G(t)y}_{D(A_p^\sigma)}\\
&=c_{\sigma,p}\left(\abs{A_p^{-\sigma}G(t)y}_{L^p(\mathcal{O})}+\abs{G(t)y}_{L^p(\mathcal{O})}\right)\\
&\le c_{\sigma,p}\left(1+|\!|A_p^{-\sigma }|\!|_{\Lin(L^p(\mathcal{O}))}\right)\abs{G(t)y}_{L^p(\mathcal{O})}\\
&\le c_{\sigma,p}\left(1+|\!|A_p^{-\sigma }|\!|_{\Lin(L^p(\mathcal{O}))}\right)\norm{G(t)}_{\Lin(\bH,L^p(\mathcal{O}))}\abs{y}_{\bH}.
\end{align*}
Hence, by Lemma \ref{gammalp}, there exists $c'>0$ such that
\[
\left|\!\left|A_p^{-\sigma}G(t)\right|\!\right|_{\gamma(\bH,L^p(\mathcal{O}))}\le c'\norm{G(t)}_{\Lin(\bH,L^p(\mathcal{O}))}
\]
and (\ref{Aqsigmaeta}) follows from (\ref{sigma-eta-2}).
\end{proof}
%


\begin{prop}\label{prop2}
Let $G(\cdot)$ be as in Lemma \ref{lem1}. Suppose further that $p,q,\beta$ and $\gamma$ satisfy
\begin{equation}\label{ineqbg}
\beta+\frac{\gamma}{2}+\frac{1}{q}+\frac{d}{p}<\frac{1}{2}.
\end{equation}
Then the mild solution to $(\ref{scp})$ exists and has almost surely trajectories in \linebreak $\mathcal{C}^\beta([0,T];\mathcal{C}^\gamma(\bar{\mathcal{O}})).$
\end{prop}

\begin{proof}
From (\ref{ineqbg}), we can find $\sigma$ such that
\[
\frac{d}{2p}<\sigma<\frac{1}{2}-\frac{1}{q}-\frac{d}{2p}-\frac{\gamma}{2}-\beta.
\]
In particular, we have $\sigma\in\bigl(\frac{d}{2p},\frac{1}{2}\bigr).$ Then, by Theorem \ref{Th:2.1} and Lemma \ref{lem1} the mild solution $u(\cdot)$ of equation (\ref{scp}) exists and is given by the stochastic convolution (\ref{2.19}). We now choose $\delta$ satisfying
\begin{equation}\label{ineqgdbs}
\frac{d}{2p}+\frac{\gamma}{2}<\delta<\frac{1}{2}-\frac{1}{q}-\beta-\sigma.
\end{equation}
The second inequality in (\ref{ineqgdbs}) and Theorem \ref{Co:2} imply that $u(\cdot)$ has trajectories almost surely in $\mathcal{C}^\beta([0,T];D(A_p^\delta)).$ The first inequality in (\ref{ineqgdbs}), Theorem 1.15.3 in \cite{triebel} and the Sobolev embedding theorem yield
\[
D(A_p^\delta)=[L^p(\mathcal{O}),D(A_p)]_\delta\subseteq H^{2\delta,p}(\mathcal{O})\hookrightarrow\mathcal{C}^\gamma(\bar{\mathcal{O}})
\]
and the desired result follows.
\end{proof}

\begin{rem}
Using results by \cite{brz1} (see e.g. Section 3.2 in \cite{brzgat}) one can prove that the same assertion in Proposition \ref{prop2} holds for $\bH=H^{\theta,2}(\mathcal{O})$ with $\theta>\frac{d}{2}+\frac{2}{q}-1,$ condition $(\ref{sigma-eta-2})$ replaced with $G(\cdot)\in\M^q(0,T;\Lin(\bH)),$ $\beta$ and $\gamma$ satisfying
\[
\beta+\frac{\gamma}{2}+\frac{1}{q}+\frac{d}{4}<\frac{1}{2}(1+\theta)
\]
and $p$ sufficiently large. In contrast, our choice of $\beta$ and $\gamma$ in Proposition \ref{prop2} depends on $d,p$ and $q$ but not on the separable Hilbert space $\bH.$
\end{rem}

\begin{ex}
Let $m>2d$ and  $g:\Omega\times[0,T]\times \mathcal{O}\to\R$ be jointly measurable and bounded with respect to $\xi\in\mathcal{O}$ such that $g(\omega,t,\cdot)\in L^m(\mathcal{O})$ for each $(t,\omega)\in[0,T]\times\Omega,$ and the map
\[
[0,T]\times\Omega\ni(t,\omega)\mapsto g(\omega,t,\cdot)\in L^m(\mathcal{O})
\]
is an $\mathds{F}$-progressively measurable process and belongs to $\mathcal{M}^q(0,T;L^m(\mathcal{O})),$ with $q$ sufficiently large so that
\[
\frac{d}{m}+\frac{1}{q}<\frac{1}{2}.
\]
Let $\theta\in\bigl(\frac{d}{m}+\frac{d-1}{2}+\frac{1}{q},\frac{d}{2}\bigr)$ also be fixed, and let $w(\cdot)$ be a cylindrical Wiener process with Cameron-Martin space $\bH=H^{\theta,2}(\mathcal{O}).$ We consider the following linear stochastic PDE on $[0,T]\times\mathcal{O}$ with zero Dirichlet-type boundary conditions and perturbed by ``colored" additive noise,
\begin{align}
\frac{\partial u}{\partial t}(t,\xi)+(\mathcal{A}u(t,\cdot))(\xi)&=g(t,\xi)\,\frac{\partial w}{\partial t}(t,\xi), \ \ \mbox{ on} \ [0,T]\times\mathcal{O}\notag\\
u(t,\xi)&=0, \ \ \ \ \ \ \ \ \ \ \ \ \ \ t\in (0,T], \ \xi\in\partial \mathcal{O}\label{spde1}\\
u(0,\cdot)&=0, \ \ \ \ \ \ \  \ \ \ \ \ \ \ \xi\in\mathcal{O}.\notag
\end{align}
\end{ex}

\begin{theorem}
Suppose $\beta$ and $\gamma$ satisfy
\begin{equation}\label{ineqtheta}
\beta+\frac{\gamma}{2}<\theta+\frac{1}{2}-d\left(\frac{1}{2}+\frac{1}{m}\right)-\frac{1}{q}.
\end{equation}
Then equation $(\ref{spde1})$ has a mild solution with trajectories almost surely in\linebreak $\mathcal{C}^\beta([0,T];\mathcal{C}^\gamma(\bar{\mathcal{O}})).$
\end{theorem}
\begin{proof}
We formulate equation $(\ref{spde1})$ as an evolution equation in $L^p(\mathcal{O})$ with $
\frac{1}{p}:=\frac{1}{2}-\frac{\theta}{d}+\frac{1}{m}.$ By the Sobolev embedding theorem, we have $\bH=H^{\theta,2}(\mathcal{O})\hookrightarrow L^r(\mathcal{O})$ continuously for $\frac{1}{r}:=\frac{1}{p}-\frac{1}{m}=\frac{1}{2}-\frac{\theta}{d}.$ Let $i_{\theta,r}$ denote this embedding. For each $(t,\omega)\in [0,T]\times\Omega,$ we define the Nemytskii multiplication operator $G(t,\omega)$ as
\[
\left(G(t,\omega)y\right)(\xi):=g(\omega,t,\xi)i_{\theta,r}(y)(\xi), \ \ \xi\in \mathcal{O}, \ \ y\in \bH.
\]
By the assumptions on $g$ and H\"{o}lder's inequality, it follows that $G(\cdot)$ is a well defined $\Lin(\bH,L^p(\mathcal{O}))$-valued process and belongs to $\mathcal{M}^q(0,T;\Lin(\bH,L^p(\mathcal{O}))).$ From condition (\ref{ineqtheta}), our choice of $p$ satisfies (\ref{ineqbg}). The desired result follows from Proposition \ref{prop2}.
\end{proof}

\begin{ex}[Fractional powers of elliptic operators]
Proposition \ref{prop2} can be easily generalized to incorporate stochastic Cauchy problems in $L^p(\mathcal{O})$ of the form \begin{equation}\label{scp2}
\begin{split}
du(t)+A_p^{\alpha/2}u(t)\,dt&=G(t)\,dW(t), \ t\in[0,T],\\
u(0)&=0.
\end{split}
\end{equation}
with $\alpha\in(0,2].$ Indeed, notice that $A_p^{\alpha/2}\in\BIP^-(\pi/2,\bE)$ for $\alpha\in (0,2].$ Let $G(\cdot)$ be as in Lemma \ref{lem1}, and suppose $p,q,\beta$ and $\gamma$ satisfy
\begin{equation}\label{ineqbg2}
\beta+\frac{1}{q}+\frac{1}{\alpha}\left(\gamma+\frac{2d}{p}\right)<\frac{1}{2}.
\end{equation}
Choose $\sigma$ such that
\[
\frac{d}{\alpha p}<\sigma<\frac{1}{2}-\frac{1}{\alpha}\left(\frac{d}{p}+\gamma\right)-\frac{1}{q}-\beta.
\]
In particular, we have $\frac{\alpha\sigma}{2}\in\bigl(\frac{d}{2p},\frac{1}{2}\bigr).$ Then, by Theorem \ref{Th:2.1} and Lemma \ref{lem1}, the mild solution $u(\cdot)$ of equation (\ref{scp}) exists. 
We now choose $\delta$ satisfying
\begin{equation}\label{ineqgdbs2}
\frac{1}{\alpha}\left(\frac{d}{p}+\gamma\right)<\delta<\frac{1}{2}-\frac{1}{q}-\beta-\sigma.
\end{equation}
The second inequality in (\ref{ineqgdbs2}) and Theorem \ref{Co:2} imply that $u(\cdot)$ has trajectories almost surely in $\mathcal{C}^\beta([0,T];D(A_p^{\alpha\delta/2})).$ The first inequality in (\ref{ineqgdbs2}) and the Sobolev embedding theorem imply that $u(\cdot)$ has trajectories almost surely in $\mathcal{C}^\beta([0,T];\mathcal{C}^\gamma(\bar{\mathcal{O}})),$ and the same conclusion of Proposition \ref{prop2} follows.
\end{ex}

\section*{Acknowledgement}
The author thanks the anonymous referee for pointing out that the main result can be easily generalized to the case of fractional powers of elliptic operators.

\bibliographystyle{imsart-nameyear}



\end{document}